\documentclass[twoside,12pt,leqno]{amsproc}
\usepackage{amssymb,stmaryrd,latexsym,enumerate,tikz,booktabs}
\usetikzlibrary{patterns}
\usepackage[pagebackref]{hyperref}
\usepackage{amsrefs}    
\usepackage{etoolbox}   
\usepackage{mathtools}  
\hypersetup{citecolor=purple, linkcolor=blue, colorlinks=true}
\usepackage{bookmark}
\numberwithin{table}{section}

\theoremstyle{plain}
\newtheorem{theorem}{Theorem}[section]
\newtheorem{lemma}[theorem]{Lemma}
\newtheorem{proposition}[theorem]{Proposition}
\newtheorem{corollary}[theorem]{Corollary}

\theoremstyle{definition} 
\newtheorem{definition}[theorem]{Definition}
\newtheorem{remark}[theorem]{Remark}
\newtheorem{question}[theorem]{Question}
\newtheorem{example}[theorem]{Example}
\AtEndEnvironment{remark}{\null\hfill$\Diamond$}

\renewcommand{\geq}{\geqslant}
\renewcommand{\leq}{\leqslant}
\renewcommand{\ge}{\geqslant}
\renewcommand{\le}{\leqslant}
\newcommand{\lhdeq}{\trianglelefteqslant}    
\newcommand{\coloneq}{\vcentcolon=}      

\newcommand{\Aut}{\mathrm{Aut}}
\newcommand{\B}{\mathcal{B}}
\newcommand{\C}{\mathrm{C}}

\newcommand{\End}{\textup{End}}

\newcommand{\PP}{\textup{P\kern-1ptP}}

\newcommand{\Sym}{\mathrm{Sym}}

\newcommand{\W}{\mathcal{W}}
\newcommand{\Z}{\mathrm{Z}}
\newcommand{\ZZ}{\mathbb{Z}}

\makeatletter        
\def\@adminfootnotes{%
  \let\@makefnmark\relax  \let\@thefnmark\relax
  \ifx\@empty\@date\else \@footnotetext{\@setdate}\fi
  \ifx\@empty\@subjclass\else \@footnotetext{\@setsubjclass}\fi
  \ifx\@empty\@keywords\else \@footnotetext{\@setkeywords}\fi
  \ifx\@empty\thankses\else \@footnotetext{%
    \def\par{\let\par\@par}\@setthanks}%
  \fi}\makeatother   

\usepackage{todonotes}

\begin{document}

\hyphenation{}

\title[Groups satisfying the functional equation \texorpdfstring{$f(xk)=xf(x)$}{}]{The groups \texorpdfstring{$G$}{G} satisfying a functional equation\\ \texorpdfstring{$f(xk)=xf(x)$}{f(xk)=xf(x)} for some \texorpdfstring{$k \in G$}{k in G}}
\author[Bernhardt, Boykett, Devillers, Flake, Glasby]{Dominik Bernhardt, Tim Boykett, Alice Devillers, Johannes Flake, S.\,P. Glasby}

\address[Bernhardt, Flake]{
Algebra and Representation Theory, RWTH Aachen University,
Pontdriesch 10-16,
52062 Aachen, Germany.
Email: {\tt \{bernhardt,\;flake\}@art.rwth-aachen.de}}

\address[Boykett]
{Time's Up Research, Industriezeile 33b, Linz, 4020 Austria\newline and Institute for Algebra, Johannes Kepler University Linz, 4040 Austria\newline
and Design Investigations, University of Applied Arts, Vienna.\newline
Email: {\tt tim@timesup.org, tim.boykett@jku.at }}

\address[Devillers, Glasby]{
Centre for Mathematics of Symmetry and Computation,
University of Western Australia,
35 Stirling Highway,
Perth 6009, Australia.\newline
 Email: {\tt \{Stephen.Glasby,\;Alice.Devillers\}@uwa.edu.au; WWW: \kern-1pt\href{http://www.maths.uwa.edu.au/~glasby/}{https://stephenglasby.github.io/}}
}
\date{\today\hfill 2010 Mathematics subject classification:
  20D15, 20E34, 20F10}

\begin{abstract}
  We study the groups $G$  with the curious property that there exists an element $k\in G$ and a function
  $f\colon G\to G$ such that $f(xk)=xf(x)$ holds for all $x\in G$. This property arose from the study of near-rings and input-output automata on groups. We call a group with this property a \emph{$J$-group}. Finite $J$-groups must have odd order, and hence are solvable. 
  We prove that every finite nilpotent group of odd order is a $J$-group if its nilpotency class $c$ satisfies $c\le6$. 
  If $G$ is a finite $p$-group, with $p>2$ and $p^2>2c-1$, then we prove that $G$ is $J$-group. Finally, 
  if $p>2$ and $G$ is a regular $p$-group or, more generally, a power-closed one (i.e., in each section and for each $m\geq1$ the subset of $p^m$-th powers is a subgroup), then we prove that $G$ is a $J$-group.
\end{abstract}

\maketitle

\section{Introduction}\label{S1}

Investigations into near-rings and state automata on groups led to a definition of a class of groups with the following curious (and interesting) property.

\begin{definition}\label{D0}
A group $G$ is a \emph{$J$-group} if there exists an element $k\in G$ and a
function $f\colon G\to G$ satisfying
\begin{equation}\label{definingidentity}
f(xk)=xf(x) \quad\text{ for all }x\in G. 
\end{equation}
\end{definition}
Instead of $f(xk)=xf(x)$ in~\eqref{definingidentity}, we could consider three other identities, namely $f(xk)=f(x)x$, $f(kx)=xf(x)$ or $f(kx)=f(x)x$. These yield four different conditions for a  non-abelian group.
It is surprising, but comforting, that these four conditions are equivalent when the group is non-abelian, see Remark~\ref{R:equiv}.

We view the element $k\in G$ in~\eqref{definingidentity} as a \emph{constant}, and we call it a {\it witness} for the group~$G$. If $k$ has finite order~$n$, then we prove that $G$ is a $J$-group with witness $k$
if and only if $\prod_{i=1}^n xk^{n-i}=1$ holds for all $x\in G$. It is interesting that from such a $k$, a suitable
function $f\colon G\to G$ satisfying~\eqref{definingidentity} can be constructed, see Lemma~\ref{L0}.
In particular, if $G=\langle k\rangle$ has order~$n$, then $\prod_{i=1}^n xk^{n-i}=1$ holds if and only if $k^{n(n+1)/2}=1$, that is, if and only if $n$ is odd. We prove in
Corollary~\ref{cor:oddorder} that no group of even order is
a $J$-group. 

Since a finite $J$-group must have odd order by Corollary~\ref{cor:oddorder}, it is necessarily solvable. However, not all groups of odd order are $J$-groups
by Remark~\ref{R3}. We ask: 
\begin{question}\label{Conj}
  Is every finite nilpotent group of odd order a $J$-group?
\end{question}

Our main results use the following definitions. The \emph{lower central series} for a group $G$ is defined recursively by $\gamma_1(G)=G$ and
$\gamma_{i+1}(G)=[G,\gamma_i(G)]$ where $[X,Y]$ is the subgroup $\langle [x,y]\mid x\in X, y\in Y\rangle$ and $[x,y]=x^{-1}y^{-1}xy$. We say that $G$ is
\emph{nilpotent of class~$c$} if $c\ge 0$ is minimal such that $\gamma_{c+1}(G)=1$. (Every finite $p$-group is nilpotent.) A $p$-group $G$ is called \emph{power-closed}
if in each section of $G$ and for each $m\ge1$, products of $p^m$-th powers are $p^m$-th powers.

Our three main results (Theorems~\ref{T7},~\ref{T8} and~\ref{T:power-closed}) are as follows. First, if $G$ is a
finite $p$-group (with $p$ odd) of class~$c$ satisfying $p^2>2c-1$, then $G$ is a $J$-group, and if $p>2c-1$, then every
element of $G$ of maximal order is a witness (that is, a possible $k$). The fact that a witness need not have 
maximal order (see Example~\ref{nonbigwitness}) seems to complicate the theory. 
Second, if $G$ is a $p$-group of class~$c$ and $\gamma_d(G)$ contains an element of maximal order, where $d\ge c-5$,
then $G$ is a $J$-group. Third, every power-closed $p$-group (with $p$ odd) is a $J$-group. Even with these results, Question \ref{Conj} remains open in its general form.

Basic properties of $J$-groups are proved in Sections~\ref{S2} and~\ref{S3}. Our main theorems are proved in Section~\ref{S4} and
we study non-nilpotent $J$-groups and some number theoretic conditions in Section~\ref{S5}.

We now say a few words on the origin of this problem.
Given a module $M$ over a principal ring, Liebert~\cite{liebert73} studied the Jacobson radical of the endomorphism ring $\End(M)$, denoted ${\mathcal J}(\End(M))$.
He proved in ~\cite{liebert73}*{Lemma~2.1} that $K(M)\subseteq{\mathcal J}(\End(M))\subseteq H(M)$ holds where $K(M)$ and $H(M)$ are certain 2-sided ideals of $\End(M)$. 
Property~\eqref{definingidentity} arose (in an additive form) in~\cite{BW}*{Definition~8},  from the study of the 2-Jacobson radical ${\mathcal J}_2$
(hence the name $J$-group\footnote{Note that in~\cite{BW} the groups were said to
have Property~X (for want of a better name) instead of being called $J$-groups.}) of left near-rings associated with input-output
automata on groups. The connection with near-rings and the motivation for our work is described in \cite{Paper2}.
Roughly speaking, Theorem~7 of~\cite{BW} defines analogous near-ring ideals $K$ and $H$ and
proves that $K\subseteq {\mathcal J}_2\subseteq H$, and~\cite{BW}*{Theorem~10} shows that $(K=){\mathcal J}_2=H$ holds if $G$ is a $J$-group, where the first
equality follows from the proof.
It is not known whether the only groups for which $K={\mathcal J}_2=H$ holds are $J$-groups, but this is an obvious question. For precise definitions, and further context, we refer the reader to \cite{Paper2}.

\section{Witnesses}\label{S2}

In this section we prove that a function~$f$ satisfying~\eqref{definingidentity} can
be constructed from the constant $k$. Recall that $k$ is called a \emph{witness}.
We also prove that finite $J$-groups are solvable, and we establish properties
which yield fast algorithms for testing whether a given finite
group is a $J$-group. {\sc Magma}~\cite{Magma} code for studying $J$-groups is available at~\cite{G}.

\begin{lemma}\label{L-1}
  Let $G$ be a $J$-group with function $f$ and witness $k$.
  The following equations hold for any positive integer $m$ and any $x\in G$:
  \begin{enumerate}[{\rm (a)}]
    \item $f(xk^m)=xk^{m-1}xk^{m-2}x\cdots xk^2xkxf(x)
    =x^m(k^{m-1})^{x^{m-1}}\cdots(k^2)^{x^2}k^xf(x),$ and 
     \item $f(xk^{-m})=k^{m}x^{-1}k^{m-1}x^{-1}\cdots x^{-1}k^2x^{-1}kx^{-1}f(x)
    =x^{-m}(k^{m})^{x^{-m}}\cdots(k^2)^{x^{-2}}k^{x^{-1}}f(x).$ 
    \end{enumerate}    
\end{lemma}

\begin{proof}
We use induction on $m$ to prove parts~(a) and (b). 

(a)~The identity $f(xk^m)=xk^{m-1}xk^{m-2}x\cdots xk^2xkxf(x)$ is true
when $m=1$ by~\eqref{definingidentity}. Suppose $m>1$.
Using~\eqref{definingidentity} and induction yields
\[
  f(xk^m)=xk^{m-1}f(xk^{m-1})=xk^{m-1}xk^{m-2}x\cdots xk^2xkxf(x),
\]
as desired. The second equation in~(a)  follows by collecting powers
of $x$ to the left.  

(b)~Observe that 
  $f(x)=f(xk^{-1}k)=xk^{-1}f(xk^{-1})$, and so $f(xk^{-1})=kx^{-1}f(x)$. Hence the identity
  $f(xk^{-m})=k^{m}x^{-1}k^{m-1}x^{-1}\cdots x^{-1}k^2x^{-1}kx^{-1}f(x)$ is true
  when $m=1$. Suppose $m>1$. By~\eqref{definingidentity} we have
  $f(xk^{-m}k)=xk^{-m}f(xk^{-m})$, and by induction
  \[
  f(xk^{-m})=k^mx^{-1}f(xk^{-(m-1)})= k^mx^{-1}k^{m-1}x^{-1}k^{m-2}x^{-1}\cdots x^{-1}k^2x^{-1}kx^{-1}f(x).
  \]
  Hence the first equation of part~(b) follows.  The second equation follows by collecting powers of $x^{-1}$ to the left.
\end{proof}

We prove that groups with an element of infinite order are $J$-groups.
We employ the convention that an empty product such as $\prod_{j=1}^0$ or
$\prod_{j=0}^{-1}$ equals 1.

\begin{proposition}\label{P-inforder}
Let $G$ be a group with an element $k$ of infinite order. Then $G$ is a $J$-group with constant $k$.
\end{proposition}

\begin{proof}
  Let $k$ be the element of infinite order, and let $\{c_i \mid i \in I\}$
  be a set of left coset representatives for $K=\langle k\rangle$ in $G$. (Note the existence of a transversal $\{c_i\mid i\in I\}$ requires the Axiom of Choice since $G$ is infinite.)  
  Then every element in $G$ can be written uniquely as $c_ik^m$ for some
  $i\in I$ and $m\in\ZZ$.
  Guided by Lemma~\ref{L-1}, we define a function $f$ as follows:
  \[
     f(c_ik^m)=\begin{cases} \prod_{j=1}^m c_ik^{m-j}&\text{ if $m\ge0$, and}\\
     \prod_{j=0}^{-m-1} k^{-m-j}c_i^{-1}&\text{ if $m<0$.} \end{cases}
  \]
We now show that $k$ and $f$ satisfy~\eqref{definingidentity}. 

Let $x=c_ik^m$ for some $i\in I$ and $m\in\ZZ$ and observe that $f(xk)=f(c_ik^{m+1})$.

{\sc Case} $m\ge0$.  Here $f(xk)=\prod_{j=1}^{m+1} c_ik^{m+1-j}$ and
$xf(x)=c_ik^m\prod_{j=1}^m c_ik^{m-j}$. These two expressions are equal.

{\sc Case} $m=-1$.  Here $f(xk)=f(c_i)=1$ and $xf(x)=c_ik^{-1}f(c_ik^{-1})=c_ik^{-1}kc_i^{-1}=1$.

{\sc Case} $m<-1$. Here $m+1<0$ so $f(xk)=f(c_ik^{m+1})=\prod_{j=0}^{-m-2} k^{-(m+1)-j} c_i^{-1}$, while
\[
  xf(x)=c_ik^m\prod_{j=0}^{-m-1} k^{-m-j} c_i^{-1}=c_ik^mk^{-m} c_i^{-1}\prod_{j=1}^{-m-1} k^{-m-j} c_i^{-1}=\prod_{j=0}^{-m-2} k^{-m-1-j} c_i^{-1}.
\]
Thus in each case $f(xk)$ equals $xf(x)$, as desired.
\end{proof}

By Proposition~\ref{P-inforder}, it suffices to consider
\emph{torsion} groups. Our focus is on \emph{finite} groups. 
Note that some of the results below do not require the group to be a torsion group.

Denote the order of a group element $g$ by $|g|$.
We now show that the constant $k\in G$ in~\eqref{definingidentity} must satisfy equation~\eqref{productformula} below if it has finite order, and conversely if $k$ 
  satisfies~\eqref{productformula} and $|k|$ is finite, then a function $f$ can be constructed which
  makes $G$ into a~$J$-group.

\begin{lemma}\label{L0}
Let $G$ be a group with an element $k$ of finite order. Then $G$ is a $J$-group with constant $k$ if and only if  
 \begin{equation}\label{productformula}
 \prod_{i=1}^{|k|} xk^{|k|-i}=1 \qquad\textup{for all $x\in G$.}
 \end{equation}
\end{lemma}

\begin{proof} Let $n=|k|$.

Assume first that $G$ is a  $J$-group with constant $k$ and function $f$ satisfying~\eqref{definingidentity}. 
Setting $m=n$ in Lemma \ref{L-1} shows
\[
  f(x)=f(xk^n)=xk^{n-1}xk^{n-2}x\cdots xk^2xkxf(x)
      = \left(\prod_{i=1}^{n} xk^{n-i}\right)f(x).
\]
Hence $\prod_{i=1}^{n} xk^{n-i}=1$ for all $x\in G$. This proves the
forward implication.

Assume now that $\prod_{i=1}^{n} xk^{n-i}=1$ for all $x\in G$. Let us define a function $f$ satisfying~\eqref{definingidentity}. 
Choose a left transversal $S$ of $\langle k\rangle$ in $G$, that is, $G$ is
a disjoint union $G=\bigcup_{s\in S} s\langle k\rangle$. (If $G$ is infinite,
the existence of $S$ requires the Axiom of Choice.)
Since each element in $G$ can be written uniquely as $sk^m$ where $s\in S$
and $0\leq m<n$, we define
\[
  f(sk^m)\coloneq sk^{m-1}sk^{m-2}s\cdots sk^2sks=\prod_{i=1}^{m} sk^{m-i}.
\]
In particular, $f(s)=1$ for all $s\in S$.
We now prove this function satisfies $f(xk)=xf(x)$ for all $x\in G$.
Let $x=sk^m$. Observe that $xk=sk^{m+1}$ and $1\leq m+1<n+1$.

Assume first that $m+1<n$. Then
\[
  f(xk)=f(sk^{m+1})= \prod_{i=1}^{m+1} sk^{m+1-i}
  = sk^{m}\prod_{i=2}^{m+1} sk^{m+1-i}= sk^{m}\prod_{i=1}^{m} sk^{m-i}=xf(x).
\]
Now assume $m+1=n$. Then
\[
f(xk)=f(sk^{m+1})=f(s)=1
= \prod_{i=1}^{n} sk^{n-i}= sk^{n-1}\prod_{i=2}^{n} sk^{n-i}
= sk^{m}\prod_{i=1}^{m} sk^{m-i}=xf(x).
\]
This proves the reverse implication.
\end{proof}

For a given witness $k$, different left transversals in the proof of
Lemma~\ref{L0} will generally give rise to  different functions $f$.

Henceforth we suppress the role of $f$ and call $k$ a witness if it satisfies~\eqref{productformula}, yielding 
a `function-free' (equivalent) definition of a $J$-group.
\begin{definition}\label{D1}
  \begin{enumerate}[{\rm (a)}]
    \item An element $k\in G$ with finite order $|k|$ is called a \emph{witness} if
    $\prod_{i=1}^{|k|} xk^{|k|-i}=1$ for all $x\in G$. 
    \item If $k$ is a witness for $G$, we say that $(G,k)$ is a \emph{$J$-group}.
    \item If we want to emphasise the choice of function $f$, we say that $(G,f,k)$ is a \emph{$J$-group}.
  \end{enumerate}
\end{definition}

\begin{corollary}\label{cor:oddorder}
A finite $J$-group has odd order, and hence is solvable.
\end{corollary}
\begin{proof}
Let $G$ be a \emph{finite} $J$-group with witness $k$ of order $n$. By Lemma~\ref{L0}, $k$ satisfies~\eqref{productformula}. 
Setting $x=1$ in \eqref{productformula} shows $k^{\binom{n}{2}}=1$ since
  $\binom{n}{2}=(n-1)+\cdots+2+1$. Thus $n$ divides $\binom{n}{2}$, and so $n$ is odd.   
  
  Suppose an element $y$ of $G$ satisfies $y^2=1$, so $y^{-1}=y$. Then 
\begin{align*}
1=\prod_{i=1}^{n} yk^{n-i}&=yk^{n-1}yk^{n-2}\cdots yk^{\frac{n+3}2}yk^{\frac{n+1}2}yk^{\frac{n-1}2}yk^{\frac{n-3}2}  \cdots   yk^2yky\\
&=(y^{-1}k^{-1}y^{-1}k^{-2}\cdots y^{-1}k^{-\frac{n-3}2}y^{-1}k^{-\frac{n-1}2})\,y\,(k^{\frac{n-1}2}yk^{\frac{n-3}2}  \cdots   yk^2yky)\\
&=g^{-1}yg\qquad\textup{where $g=k^{\frac{n-1}2}yk^{\frac{n-3}2}  \cdots   yk^2yky$.}
\end{align*}
  This proves that $y=1$ and hence that $|G|$ is odd (by  Cauchy's Theorem). Finally, groups of odd order are solvable by the Feit-Thompson Theorem~\cite{FT}.
\end{proof}

\begin{remark}\label{R:equiv}
  A thoughtful reader may question whether there are really four
  definitions of (non-abelian) $J$-groups, and wonder whether different definitions give
  the same sets of groups. Let us temporarily talk of a $J_i$-group,
  for $i\in\{1,2,3,4\}$, to be a group with a constant $k_i\in G$
  and a function $f_i\colon G\to G$ satisfying the $i$th equation below
  \[
  f_1(xk_1)=xf_1(x),\quad f_2(xk_2)=f_2(x)x,\quad
  f_3(k_3x)=xf_3(x),\quad f_4(k_4x)=f_4(x)x
  \]
  for all $x\in G$. To see that $J_1$- and $J_4$-groups agree take $k_4=k_1^{-1}$
  and $f_4(x)=f_1(x^{-1})^{-1}$. The same argument shows that $J_2$- and
  $J_3$-groups agree.
  By an argument similar to Proposition \ref{P-inforder}, we easily see that non-torsion groups are $J_3$-groups.
  By Lemma~\ref{L0}, $G$ is a torsion
  $J_1$-group if and only if there is a $k_1\in G$ satisfying
  $\prod_{i=1}^{n_1} xk_1^{n_1-i}=1$ for all $x\in G$ where $n_1=|k_1|$.
  A similar argument shows that $G$ is a torsion $J_3$-group if and only if
  there is a $k_3\in G$ satisfying
  $\prod_{i=1}^{n_3} k_3^{n_3-i}x=1$ for all $x\in G$ where $n_3=|k_3|$.
  Taking inverses changes $\prod_{i=1}^{n_1} xk_1^{n_1-i}=1$
  to $\prod_{i=n_1}^{1} k_1^{i-n_1}x^{-1}=1$. Replacing $x$ with $y^{-1}k_1$ gives
  $\prod_{i=n_1}^{1} k_1^{i-1}y=1$ or $\prod_{j=1}^{n_1} k_1^{n_1-j}y=1$.
  However, $y$ ranges over $G$ precisely when
  $x\coloneq y^{-1}k_1$ does. Thus a $J_3$-group is a $J_1$-group (with the same witness). A similar argument shows that a $J_1$-group is a $J_3$-group. Hence the four definitions are equivalent.
\end{remark}

\begin{remark}
Since~\eqref{productformula} involves two group elements ($x$ and $k$), $k$ is a witness for $G$ if and only if it is a witness for every 2-generated subgroup 
of $G$ containing $\langle k\rangle$.
\end{remark}

Two subgroups $H,K$ of $G$
give rise to double cosets $HxK$,
whose size depends on $x\in G$,
and these induce a
partition $G=Hx_1K\sqcup\cdots\sqcup Hx_sK$
of~$G$.
We show that Eq.~\eqref{productformula} need only be tested for a set of double coset representatives, which simplifies any computer check.

\begin{lemma}\label{L11}
Fix $k\in G$ of order $n$ and let $K=\langle k\rangle$. Then
$\prod_{i=1}^n xk^{n-i}=1$ holds for all $x\in G$ if and only if it holds for $x_1,\dots,x_s$ where
$G=Kx_1 K\sqcup\cdots\sqcup Kx_s K$.
\end{lemma}

\begin{proof}
  Replacing $x$ by $xk$ or $kx$ in the equation $\prod_{i=1}^n xk^{n-i}=1$ gives
  \[
    1=\prod_{i=1}^n (xk)k^{n-i}=\left(\prod_{i=1}^n xk^{n-i}\right)^{x^{-1}}
    \quad\textup{and}\quad
    1=\prod_{i=1}^n (kx)k^{n-i}=\left(\prod_{i=1}^n xk^{n-i}\right)^{(kx)^{-1}}.
  \]
  As these equations are conjugates of the equation $\prod_{i=1}^n xk^{n-i}=1$, it suffices to choose
  $x$ from a set $\{x_1,\dots,x_s\}$ of
  $K\kern-2.5pt\setminus\kern-2.5pt G/K$ double coset representatives for $G$.
\end{proof}


  The following lemma is an easy consequence of Proposition~\ref{P-inforder} if $k$ has infinite order, and of Lemma~\ref{L0}~otherwise.

\begin{lemma}\label{L9}  If $(G,k)$ is a $J$-group and $k\in H\le G$,
    then $(H,k)$ is a $J$-group.
\end{lemma}

\begin{definition}\label{D:W}
  The set of all witnesses $k\in G$ (see Definition~\ref{D1}) is denoted $\W(G)$.
\end{definition}

The next result further limits the number of checks required to find a witness $k$.

\begin{lemma}\label{L-closedaut}
 Let $G$ be a torsion $J$-group. The set $\W(G)$ is closed under $\Aut(G)$.
\end{lemma}

\begin{proof}
  Let $k\in \W(G)$ and $\alpha\in \Aut(G)$.
  Then $n\coloneq |k|$ is finite and $\prod_{i=1}^n xk^{n-i}=1$ for all $x\in G$. Since
  $|k^\alpha|=n$, the following calculation proves that $k^\alpha\in\W(G)$:
  \[
    \prod_{i=1}^n x(k^\alpha)^{n-i}=\prod_{i=1}^n x(k^{n-i})^\alpha
    =\prod_{i=1}^n (x^{\alpha^{-1}}k^{n-i})^\alpha=\left(\prod_{i=1}^n x^{\alpha^{-1}}k^{n-i}\right)^\alpha=1^\alpha=1.\qedhere
  \]
\end{proof}

\begin{lemma}\label{L-directproduct}
Let $G_1,G_2$ be finite $J$-groups. Then $\W(G_1)\times\W(G_2)\subseteq \W(G_1\times G_2)$ and equality holds if $|G_1|$ and $|G_2|$ are coprime. Thus
$J$-groups are closed under products.
\end{lemma}
 Note that equality does not hold in general, see Remark \ref{non-equal}.
\begin{proof}
Let $k_j\in\W(G_j)$ with $|k_j|=n_j$ for $j=1,2$.   The order of $(k_1,k_2)\in G_1\times G_2$ is $n\coloneq \textup{lcm}(n_1,n_2)=r_1n_1=r_2n_2$ where $r_1,r_2$ are positive integers. Let  $x_j\in G_j$. Then 
  $\prod_{i=1}^{n_jr_j} x_jk_j^{n_jr_j-i}=(\prod_{i=1}^{n_j} x_jk_j^{n_j-i})^{r_j}=1^{r_j}=1$.
  Hence \[\prod_{i=1}^{n} (x_1,x_2)(k_1,k_2)^{n-i}=(\prod_{i=1}^{n_1r_1} x_1k_1^{n_1r_1-i},\prod_{i=1}^{n_2r_2} x_2k_2^{n_2r_2-i})=(1,1),\] and so $(k_1,k_2)\in\W(G_1\times G_2)$.
  Hence 
  $\W(G_1)\times\W(G_2)\subseteq \W(G_1\times G_2)$,
  and $(G_1\times G_2,k)$ is a $J$-group where $k=(k_1,k_2)$.
 
  Suppose now that $|G_1|$ and $|G_2|$ are coprime. Let $k=(a_1,a_2)$ be a witness for the $J$-group $G=G_1\times G_2$, we will now show that $a_j$ is a witness for $G_j$, for each $j=1,2$. Let $m_j\coloneq |a_j|$ for $j=1,2$.
  We have $|k|=\textup{lcm}(m_1,m_2)=m_1m_2$ since $m_1\mid |G_1|$, $m_2\mid |G_2|$ and $|G_1|$ and $|G_2|$ are coprime. Moreover, $\prod_{i=1}^{m_1m_2} xk^{m_1m_2-i}=1$ for all $x=(x_1,x_2)\in G$. Writing $m_1m_2/m_j$ as $m_{3-j}$, we have
  that $1=\prod_{i=1}^{m_1m_2} x_ja_j^{m_1m_2-i}=(\prod_{i=1}^{m_j} x_ja_j^{m_j-i})^{m_{3-j}}$ 
  for all $x_j\in G_j$ and $j=1,2$. Since the order of $\prod_{i=1}^{m_j} x_ja_j^{m_j-i}$ divides $|G_j|$, it is coprime with $m_{3-j}$, and so 
  $\prod_{i=1}^{m_j} x_ja_j^{m_j-i}=1$ for $j=1,2$. This shows that $a_j\in\W(G_j)$ for $j=1,2$, and therefore that $\W(G_1)\times\W(G_2)\supseteq \W(G_1\times G_2)$.
\end{proof}

\begin{corollary}\label{L10}
  Suppose there is  $k\in G$ and normal subgroups $N_1,N_2$ of $G$ such that $N_1\cap N_2=1$. If  $(G/N_1,kN_1)$ and $(G/N_2,kN_2)$
  are $J$-groups, then $(G,k)$ is a $J$-group.
\end{corollary}

\begin{proof}
  Lemma~\ref{L-directproduct} implies $(G/N_1\times G/N_2,(kN_1,kN_2))$ is a $J$-group.
  The map $g\mapsto(gN_1,gN_2)$ embeds $G$ into $G/N_1\times G/N_2$ as
  $N_1\cap N_2=1$. Now apply Lemma~\ref{L9}.
\end{proof}

\begin{lemma}\label{L6}
  If $(G,k)$ is a torsion $J$-group and $\langle k\rangle\cap N=1$ where
  $N\lhdeq G$, then $(G/N,kN)$ is a $J$-group.
\end{lemma}

\begin{proof}
  Since $\langle k\rangle\cap N=1$, the order of $kN\in G/N$ equals the
  order of $k\in G$. As~\eqref{productformula} holds for the $J$-group $G$, it holds
  in $G/N$ too (replace $x$ with $xN$ and $k$ with $kN$). By Lemma \ref{L0}, $(G/N,kN)$ is a $J$-group.
\end{proof}

\section{Finite \texorpdfstring{$J$}{}-groups}\label{S3}

Suppose henceforth that $G$ is a finite $J$-group. Thus $|G|$ is odd by Corollary~\ref{cor:oddorder}. The \emph{exponent} of $G$ is
$\exp(G)=\textup{lcm}\{|g|\mid g \in G\}$. Hence $x^{\exp(G)}=1$ for all $x\in G$.

\begin{definition}\label{Dbig}
An element $x\in G$ is called  \emph{big} if $|x|=\exp(G)$. We write
\[\B(G)=\{x\in G\mid \textup{$x$ is big}\}.\]
\end{definition}

Note that $\B(G)$ is non-empty when $G$ is a nilpotent group.

\begin{lemma}\label{L-sameexp}
Let $G$ be a finite group of odd order.
\begin{enumerate}[{\rm (a)}]
\item If $k\in\W(G)$, then $|k|=\exp(C_G(k))$. Hence if $k\in\langle g\rangle$ for some $g\in G$, then
$|g|=|k|$.
\item If $\exp(\Z(G))=\exp(G)$, then $G$ is a $J$-group, and 
  $\Z(G)\cap \W(G)=\B(\Z(G))$. In particular, $\B(\Z(G))\subseteq \W(G)$, and if
  $G$ is abelian $\B(G)=\W(G)$.
\item If $\exp(\Z(G))<\exp(G)$, then $\Z(G)\cap \W(G)=\emptyset$. Hence if $\exp(\Z(G))<\exp(G)$, then $\W(G)\subseteq G\setminus \Z(G)$.
\end{enumerate}
\end{lemma}

\begin{proof}
(a)~Suppose that $k$ has order $n\coloneq |k|$. Now $|G|$ is odd by~Corollary~\ref{cor:oddorder}. Hence $n$ is odd and $n$ divides $\frac{n(n-1)}{2}$, so $k^{\binom{n}{2}}=1$. If $x$ centralizes $k$, then
\begin{equation}\label{E:Ck}
  \prod_{i=1}^{n} xk^{n-i}=x^n\prod_{i=1}^{n} k^{n-i}
  =x^nk^{1+2+\cdots +(n-1)}=x^nk^{\binom{n}{2}}=x^n.
\end{equation}
Suppose now that $k\in\W(G)$. Then $1=x^n$ by~\eqref{E:Ck} and hence $\exp(C_G(k))=n$. Also, if $k\in\langle g\rangle$, then $|g|$ is a multiple of $n$ and $g\in C_G(k)$ so $g^n=1$.  Therefore $|g|=|k|=n$.

(b)~Suppose that $\exp(\Z(G))=\exp(G)=n$ and $k\in\Z(G)$ has order~$n$. It follows from~\eqref{E:Ck} that $\prod_{i=1}^{n} xk^{n-i}=1$ holds for all $x\in G$.
Hence $(G,k)$ is a $J$-group and $\B(\Z(G))\subseteq\Z(G)\cap\W(G)$ holds. The reverse containment holds by part~(a).
Finally, if $G$ is abelian, then $G=\Z(G)$ and so $\B(G)=\W(G)$.

(c)~Assume $\exp(\Z(G))<\exp(G)$. If $k\in\Z(G)\cap \W(G)$, then
$1=\prod_{i=1}^{|k|} xk^{|k|-i}=x^{|k|}$ and $\exp(G)$ divides $|k|$. Thus
$\exp(G)=|k|$ and $\exp(\Z(G))=\exp(G)$, a contradiction.
\end{proof}
Lemma~\ref{L-sameexp}(b)
and Corollary~\ref{cor:oddorder}
imply that a finite abelian group is a $J$-group if and only if its order is odd, \emph{cf.} \cite{BW}*{Theorem~8}.

Moreover, Lemma~\ref{L-sameexp}(a) implies that a witness $k$ is \emph{root-free}, that is $k\in\langle g\rangle$ implies that $|k|=|g|$. Big elements are
necessarily root-free.

\begin{remark}\label{non-equal}
 By Lemma~\ref{L-sameexp}(b), $\W(G)=\B(G)$ for a finite \emph{abelian}
group of odd order. Hence if $G_1,G_2$ are abelian $p$-groups with $\exp(G_1)=\exp(G_2)=n>2$, then $G_1\times G_2$ is also abelian with exponent $n$. It follows that 
$\W(G_1)\times\W(G_2)\subsetneqq \W(G_1\times G_2)$ since $(k_1,1)\in \W(G_1\times G_2)\setminus \left( \W(G_1)\times\W(G_2)\right)$ if $k_1\in\B(G_1)$, \emph{cf.} Lemma~\ref{L-directproduct}.
\end{remark}

Our experiments lead us to ask:

\begin{question}\label{Q2}
  If $G$ is a finite nilpotent group $G$, then is $\W(G)\cap\B(G)$ non-empty?
\end{question}

\begin{corollary}\label{L8}
  Suppose that $p$ is an odd prime, and $G$ be a finite $p$-group with exponent~$p$.  Then $G$ is a $J$-group and $\Z(G)\cap \W(G)=\B(\Z(G))$.
\end{corollary}

\begin{proof}
Since $\exp(G)=p$, we have $G\ne1$. A non-trivial $p$-group has non-trivial centre, so $\exp(\Z(G))=p=\exp(G)$ and the result follows from Lemma \ref{L-sameexp}(b).
\end{proof}

If $G\ne1$, then each witness $k$ is non-trivial.
Thus $\W(G)\subseteq G\setminus\{1\}$ with equality if $G$ is
elementary abelian. Surprisingly,
equality also holds for $p$-groups of exponent $p$ when $p$ is `large' relative to the nilpotency class, see Corollary \ref{E5}(a). For instance this holds for the extraspecial group $p_{+}^{1+2m}$ when $p\ge5$, see Remark~\ref{re:ES}.

The following lemma generalises Lemma~\ref{L-sameexp}(b) by setting $N=G$.

\begin{lemma}\label{L7bis}
  Let $G$ be a $p$-group where $p>2$. Suppose that $N\lhdeq G$, $k\in G$, $kN\in G/N$ has order $m$, $k^m\in\B(N)\cap\Z(N)$, and $(G/N,kN)$ is a $J$-group. Then $(G,k)$ is a $J$-group.
\end{lemma}

\begin{proof}
Suppose $n=|k|$ and $kN\in G/N$ has order $m$. Then $r\coloneq n/m\in\mathbb{Z}$. Let us fix an arbitrary element $x$ in $G$. Fix  $(a_1,\dots,a_m)\in\mathbb{Z}^m$ and set
\[
  E(a_1,\dots,a_m) \coloneq  xk^{a_1} \cdots xk^{a_m}.
\]

We claim that $E(a_1,\dots,a_m)$ lies in $N$ provided $a_{i+1}\equiv a_i-1 \pmod m$ for all $1\leq i< m$. Indeed, $E(m-1,m-2,\dots,1,0)$ lies in $N$, since $kN$ is a
witness of $G/N$ of order $m$. Furthermore, if $E(a_1,\dots,a_m)\in N$ for some $m$-tuple $(a_1,\dots,a_m)\in\mathbb{Z}^m$, then
\[
  E(a_2,\dots,a_m,a_1) = (xk^{a_1})^{-1} E(a_1,\dots,a_m) (xk^{a_1})\in N
\] 
and
\[
  E(a_1,\dots,a_{m-1},a_m\pm m) = E(a_1,\dots,a_{m-1},a_m) k^{\pm m}\in N,
\]
 because $N$ is normal and $k^m\in N$. Hence, we can rotate the tuples $(a_1,\dots,a_m)$ and modify arbitrary entries by multiples of $m$, without changing the membership status of the corresponding $E(a_1,\dots,a_m)$ with respect to $N$. This proves the claim.

Suppose $(a_1,\dots,a_m), (b_1,\dots,b_m)\in\mathbb{Z}^m$ satisfy $a_{i+1}\equiv a_i-1\pmod m$ and $b_{i+1}\equiv b_i-1\pmod m$ for all $1\leq i<m$, and $b_1\equiv a_m-1\pmod m$. Since $k^m\in\Z(N)$, we have
\begin{align}
E(a_1,\dots,a_i+m,\dots,&a_m) E(b_1,\dots,b_m)\notag \\
&=xk^{a_1} \cdots xk^{a_i} k^m E(a_{i+1},\dots,a_m,b_1,\dots b_i) xk^{b_{i+1}} \cdots xk^{b_m}\notag\\
&= xk^{a_1} \cdots xk^{a_i} E(a_{i+1},\dots,a_m,b_1,\dots b_i)  k^m xk^{b_{i+1}} \cdots xk^{b_m}\notag\\
&= E(a_1,\dots,a_m) E(b_1,\dots,b_i+m,\dots,b_m)\label{E:m2}
\end{align}
as $E(a_{i+1},\dots,a_m,b_1,\dots b_i)\in N$. Also, 
\begin{equation}\label{E:n2}
  E(a_1+n,\dots,a_m+n) = E(a_1,\dots,a_m),
\end{equation}
for all $(a_1,\dots,a_m)\in\mathbb{Z}^m$, since $k^n=1$.

We use~\eqref{E:m2} to move multiples of $m$ below as far to the right as possible before using~\eqref{E:n2}:
\begin{align*}
xk^{n-1} \cdots xk x
 &=\kern-1pt E(rm-1,\dots,rm-m) \cdots E(2m-1,\dots, 2m-m)\cdot E(m-1,\dots,0) \\
 &=\kern-1pt E(-1,\dots,-m)^{r-1} E((1+2+\dots+r)m - 1,\dots ,(1+2+\dots+r)m - m) \\
 &=\kern-1pt E(-1,\dots,-m)^r,
\end{align*}
where we use that
$
(1+2+\dots+r)m = \frac{(r+1)rm}{2}=\frac{(r+1)n}2
$ is a multiple of $n$ since the $p$-power $r$ is odd. Finally, $E(-1,\dots,-m)\in N$ and $|k^m|=r=\exp(N)$, so $E(-1,\dots,-m)^r=1$. This proves that $\prod_{i=1}^nxk^{n-i}=1$ for all $x\in G$. Thus $k$ is a witness for $G$.
\end{proof}

Our insights were informed by computations which led to 
Question~\ref{Conj}.

\begin{remark}\label{R3}
A suit of {\sc Magma}~\cite{Magma} computer programs available at~\cite{G} 
were used to
investigate properties of $J$-groups and formulate conjectures. The smallest
example of a group that is not a $J$-group is the meta-cyclic group
$M=\langle a,b\mid a^3=b^7=1,\ b^a=b^2\rangle$ of order~21.
This example was first found by Vinay Madhusudanan in 2015. Our programs show that
the smallest five groups that are not $J$-groups are identified in {\sc GAP}~\cite{GAP} and {\sc Magma}~\cite{Magma} by the tuples
$\langle\textup{order,\,number}\rangle\in\{\langle21,1\rangle, \langle39,1\rangle, \langle55,1\rangle, \langle57,1\rangle, \langle63,1\rangle\}$. The class of $J$-groups is not closed under quotient groups
or normal subgroups. For example, the group $M\times\C_3$, where $M$ is the aforementioned group of order~21, is $J$-group and the witnesses for $M\times\C_3$ comprise the 12 elements of order 21. However, $M$ is both a quotient and a normal subgroup and is not a $J$-group.
\end{remark}

\section{Many \texorpdfstring{$p$}{}-groups are \texorpdfstring{$J$}{}-groups}\label{S4}

Recall that a group $G$ is {\it nilpotent} if its lower central series
terminates in finitely many steps. Set  $\gamma_1(G)=G$ and
$\gamma_{i+1}(G)=[\gamma_i(G),G]$ for $i\ge1$. We say that $G$ is
nilpotent of class $c$ if $\gamma_c(G)\neq 1$ and  $\gamma_{c+1}(G)= 1$.


\begin{remark}\label{pgroupssufficient}
If $G$ is finite and nilpotent, then $G\cong \prod_{i=1}^r P_i$ where $P_i$ is a $p_i$-group and the primes $p_i$ are distinct~\cite{Macdonald}*{Theorem 9.08}. Thus Questions~\ref{Conj} and~\ref{Q2} reduce to the case of $p$-groups, as
$\W(G)=\prod_{i=1}^r \W(P_i)$ by Lemma~\ref{L-directproduct} and $\B(G)=\prod_{i=1}^r \B(P_i)$.
\end{remark}

In this section, we show that many classes of $p$-groups are $J$-groups. The next four subsections
consider $p$-groups $G$ with (1) $p$ large, (2) $c$  small, (3) $G$ ``powerful'', and (4) $G$ ``power-closed''. Nevertheless, the question of whether all $p$-groups with $p>2$ are $J$-groups (\emph{cf.} Question~\ref{Conj}) remains open.

We know that a $p$-group $G$ is a $J$-group if $\exp(\Z(G))=\exp(G)$ by
Lemma \ref{L-sameexp}(b). 

A natural place   to look for a counter-example, in view of our results below, would be $3$-groups of nilpotency class at least $7$ that are not 
power-closed. We used {\sc Magma}~\cite{Magma} and {\sc GAP}~\cite{GAP} to check that the $9905$ groups of order dividing $3^7$ and the $35082$ groups of order dividing $5^7$ are indeed $J$-groups and each has a big witness, that is $\W(G)\cap\B(G)\ne\emptyset$. We also used the {\sc Magma} programs~\cite{G} to check that the $1396077$ groups of order $3^8$ are $J$-groups. The theorems in this section and Remark~\ref{re:c=7} suggest that a counter-example to Question~\ref{Conj} may be so big that it would be very difficult  to check it has no witnesses.

\subsection{Large \texorpdfstring{$p$}{$p$}}
We use a left-normed convention, so $[a_1,a_2,a_3,\dots,a_n]$ is shorthand
for $[\dots[[a_1,a_2],a_3]\dots,a_n]$.
For a complex commutator of $x$ and $k$ (see~\cite{HallP}*{p.\,43}), recall that its {\it weight} is the number of $x$ and $k$. 
It will be convenient to use the following definition.
\begin{definition}
For a complex commutator of $x$ and $k$, its {\it load} is the sum of the number of $x$ and twice the number of $k$ appearing in it.
\end{definition} 

So for instance $[k,x,k,[k,x]]$ has weight $5$ and load $8$.

The following lemma is adapted from Philip Hall \cite{HallP}*{Theorem 3.1}.
\begin{lemma}\label{L-Hall}
 For any two elements $x$ and $k$ of any group $G$, let the formally distinct complex commutators of $x$ and $k$ be $R_1=x$, $R_2=k$, $R_3,\dots, R_i,\dots$ arranged in order of increasing weights, the order being otherwise arbitrary.
 Let $m_i$ be the load of $R_i$.
 Then there exists a series of integer-valued polynomials $f_1(n),f_2(n),f_3(n),\dots$ each of the form 
 \[
   f_i(n)=a_{i1} \binom{n}{1}+a_{i2} \binom{n}{2}+\cdots +a_{im_i}
   \binom{n}{m_i} \text{ where each $a_{ij}$ is a non-negative integer},
 \]
 with the property that for all $x$ and $k$ and all positive integers $n$:
 \[
 \prod_{i=1}^n xk^{n-i}=R_1^{f_1(n)}R_2^{f_2(n)}\cdots R_i^{f_i(n)}\cdots.
 \]
 The right hand infinite product is interpreted to mean that, if $R_\lambda$, $\lambda=\lambda(c)$, is the last term which is of weight less than $c$ in $x$ and $k$, then 
 \[
 \prod_{i=1}^n xk^{n-i}\equiv R_1^{f_1(n)}R_2^{f_2(n)}\cdots R_\lambda^{f_\lambda(n)} \pmod {\gamma_c(\langle x,k\rangle)},
 \]
this being true for $c=1,2,3,\dots$.
 \end{lemma}

Observe that we never introduce new $x$ and $k$ during the collection process described below, so $f_1(n)=n=\binom{n}{1}$ and  $f_2(n)=\binom{n}{2}$. 
 
\begin{proof}
The proof of this result involves a collection process similar to that described in \cite{HallP}*{Theorem 3.1} but applied to the word $\prod_{i=1}^n xk^{n-i}$  rather than the word $(xk)^n$. Given complex commutators $R$ and $S$, collection replaces a subword $SR$ with $RS[S,R]$. The subtlety involves the {\it order} in which the commutators are collected.
To minimize duplication, we
refer the reader to Ph.~Hall's nicely written proof~\cite{HallP}*{Theorem 3.1}, and we focus on
where our proof differs, namely the different starting word and different~labelling.

At stage $i$ we collect the commutator $R_i$ to the left by repeatedly replacing $SR_i$ with $R_iS[S,R_i]$, starting with the leftmost instance of $R_i$, then the next one, etc.
Together with the collecting process comes a labeling process: each $R_i$ as it arose during the collection process is assigned a label 
$\Lambda_i=(\lambda_1,\lambda_2,\dots, \lambda_{m_i})$ (each $\lambda_j$ is an integer between $1$ and $n$) in such a way that distinct $R_i$'s are assigned different labels.
Recall that $m_i$ is the load of $R_i$, i.e. the sum of the number of $x$ and twice the number of $k$ appearing in $R_i$.

\makeatletter
\def\smallunderbrace#1{\mathop{\vtop{\m@th\ialign{##\crcr
   $\hfil\displaystyle{#1}\hfil$\crcr
   \noalign{\kern3\p@\nointerlineskip}%
   \tiny\upbracefill\crcr\noalign{\kern3\p@}}}}\limits}
\makeatother
At the start we have the expression
\[
\prod_{i=1}^n xk^{n-i}=x\underbrace{k\cdots k}_{n-1}x\underbrace{k\cdots k}_{n-2}x\cdots x\smallunderbrace{kk}_{2}x\smallunderbrace{k}_{1}x,
\]
which contains $n$ elements $x$ and $\binom{n}{2}$ elements $k$.
We assign to each $x$ the label $(\lambda_1)$ ($1\leq\lambda_1 \leq n$) in increasing order (from left to right) and to each   $k$ the label $(\lambda_1,\lambda_2)$ ($1\leq \lambda_1<\lambda_2\leq n$) in increasing lexicographic order.
To summarise, the labels (written as exponents) look like this:
\[
x^{(1)}\underbrace{k^{(1,2)}\cdots k^{(1,n)}}_{n-1}x^{(2)}\underbrace{k^{(2,3)}\cdots k^{(2,n)}}_{n-2}x^{(3)}\cdots x^{(n-2)}\underbrace{k^{(n-2,n-1)}k^{(n-2,n)}}_{2}x^{(n-1)}\underbrace{k^{(n-1,n)}}_{1}x^{(n)}.
\]

If at some stage we replace $\cdots RS\cdots$ by $\cdots SR[R,S]\cdots$ where if $R=R_i$ has label $(\lambda_1,\lambda_2,\dots, \lambda_{m_i})$ and $S=R_j$ has label $(\lambda'_1,\lambda'_2,\dots, \lambda'_{m_j})$, then $S$ and $R$ keep their labels and $[R,S]$ gets assigned the label $(\lambda_1,\lambda_2,\dots, \lambda_{m_i},\lambda'_1,\lambda'_2,\dots, \lambda'_{m_j})$. Note that $m_i+m_j$ is indeed the load of  $[R,S]$. As shown in \cite{HallP}*{p.\;69}, this process  indeed assigns different labels to different instances of a given commutator $R_i$.

To compute the number $f_i(n)$, we want to count how many different labels appear for $R_i$ at the end of the collection process for $R_i$. Hall showed \cite{HallP}*{Theorem 3.25} that if the conditions are of a certain type he calls $(\Pi)$ (that is only involve inequalities, equalities, ``and", and ``or" in terms of the label elements) for the $m_i$ label elements, then the number $f_i(n)$ is of the form
\[
f_i(n)=a_{i1} \binom{n}{1}+a_{i2} \binom{n}{2}+\cdots +a_{im_i} \binom{n}{m_i} \text{ where $a_{ij}\in\ZZ, a_{ij}\ge0$ for each $j$},
\]
as desired. He uses two types of conditions (existence and precedence): $E_i$ is a condition (on its label) for $R_i$ to exist, and $P_{ij}$ is a condition for an $R_i$ to be to the left of an $R_j$ at any stage where $R_i$ and $R_j$ have not been collected yet (this does not change as long as they both are not collected, so the notation is not ambiguous).
Moreover he shows \cite{HallP}*{Section 3.3} that the induction process (collection and labeling) transforms existence and precedence conditions of type $(\Pi)$ into other existence and precedence conditions of type $(\Pi)$. Therefore, the only thing we need to show to finish the proof is that the initial existence and precedence conditions are of type $(\Pi)$. Recall that $R_1=x$ and $R_2=k$ where we give labels $(\lambda_1)$ to $x$ and $(\lambda_1,\lambda_2)$ to $k$. There are
six initial conditions $E_1$, $E_2$, $P_{11},P_{12},P_{21} ,P_{22}$, and they can be expressed using ``$<$", ``$=$",
``or", ``and" as follows:
\begin{itemize}
\item[] $E_1$ is the empty condition;
\item[] $E_2: \lambda_1<\lambda_2$;
\item[] $P_{11}: x^{(\lambda_1)}\text{ precedes } x^{(\mu_1)}$ if and only if $\lambda_1<\mu_1$;
\item[] $P_{12}: x^{(\lambda_1)}\text{ precedes } k^{(\mu_1,\mu_2)}$ if and only if $\lambda_1<\mu_1 \textup{ or } \lambda_1=\mu_1$;
\item[] $P_{21}: k^{(\lambda_1,\lambda_2)}\text{ precedes } x^{(\mu_1)}$ if and only if $\lambda_1<\mu_1$; 
\item[] $P_{22}: k^{(\lambda_1,\lambda_2)}\text{ precedes } k^{(\mu_1,\mu_2)}$ if and only if $\lambda_1<\mu_1 \textup{ or } (\lambda_1=\mu_1 \textup{ and } \lambda_2<\mu_2)$.
\end{itemize}
Therefore these initial conditions are indeed of type $(\Pi)$. This concludes the proof.
\end{proof}

\begin{remark}
  Lemma~\ref{L-Hall} can easily be generalised to collecting any expression in two variables $x$ and $k$ such that the expression has a labelling of type $(\Pi)$. If, in this labelling, $x$'s get $i$ labels and $k$'s get $j$ labels, then the load of a commutator $R$ is $i$ times the number of $x$ in $R$ plus $j$ times the number of $k$ in $R$, and the proof is exactly the same. The original expression in \cite{HallP}*{Theorem~3.2} has the form $(xk)^n$ which has a labelling of type~$(\Pi)$ with $i=j=1$. 
\end{remark}

For a prime $p$, the \emph{$p$-part} of a positive integer $j$ is the largest $p$-power that divides $j$.

\begin{lemma}\label{Kummer}
Suppose that $p$ is a prime, $n=p^r$ is a $p$-power and $j\in\{1,2,\dots,n\}$. If $j$ has $p$-part $p^\ell$, then  $\binom{n}{j}$ has $p$-part $n/p^\ell$.
\end{lemma}

\begin{proof}
Kummer's Theorem ~\cite{Kummer} says that the $p$-part of the binomial $\binom{n}{j}$ is $p^c$ where $c$ is the number of carries involved in adding the $p$-adic expansions of $n-j$ and $j$. Since $n=p^r$ is a $p$-power, there are precisely $c=r-\ell$ carries, and the result follows.
\end{proof}

\begin{theorem}\label{T7}
 Let $G$ be a $p$-group of nilpotency class  $c$, with $p$ odd.
 \begin{enumerate}[{\rm (a)}]
\item If $p$ is greater than $2c-1$, then  $\B(G)\subseteq\W(G)$ and $G$ is a $J$-group.
\item If $p^2$ is greater than $2c-1$, then  $\B(\gamma_d(G))\subseteq\W(G)$, where $d$ is the largest integer such that $\exp(\gamma_d(G))=\exp(G)$, and $G$ is a $J$-group.
\end{enumerate}
\end{theorem}
\begin{proof}

Let $\exp(G)=n$ be a power of $p$. Recall that $g^n=1$ for all $g\in G$.

Both assertions are true for abelian $p$-groups by Lemma \ref{L-sameexp}(b), so we may assume $c\ge2$. We have $\gamma_{c+1}(G)= 1$, since $G$ is nilpotent of class $c$ and so by Lemma~\ref{L-Hall}
\[
\prod_{i=1}^n xk^{n-i}= R_1^{f_1(n)}R_2^{f_2(n)}\cdots R_\lambda^{f_\lambda(n)}
\]
 where $R_\lambda$ is the last commutator of weight $c$ in the list we used in Lemma~\ref{L-Hall}. 
 Moreover, we know that 
\begin{equation}\label{E:fn}
 f_i(n)=a_{i1} \binom{n}{1}+a_{i2} \binom{n}{2}+\cdots +a_{im_i} \binom{n}{m_i}
\end{equation}
where each $a_{ij}$ is a non-negative integer,
so  $f_i(n)=nP_i(n)$ where $P_i(n)$ is an integer-coefficient polynomial of degree at most $m_i-1$ divided by the factorial $m_i!$.

Let $R$ be a non-trivial complex commutator of weight~$w$. Since each element in $R$ contributes $1$ or $2$ to the load, we see $w\le\textup{load}(R)\le 2w$. However, $\textup{load}(R)=2w$ can happen only if each symbol in $R$ is $k$ but such a commutator would be trivial unless $R=R_2=k$.
Thus $w\le\textup{load}(R)< 2w$ for all $R\neq R_2$. As all the commutators $R_i$ for $1\leq i\leq \lambda$ have weight at most~$c$, we have $m_i=\textup{load}(R_i)\leq 2c-1$ for all $i$, using that $c\ge2$ for the case $i=2$.
We first prove part~(a).

(a)~Let $k\in \B(G)$, so that $|k|=n$, and assume $p>2c-1$. 

Since $m_i\le 2c-1<p$ for $1\le i\le \lambda$, it follows that the $p$-part of $j$ is 1 for $1\le j\le m_i$, and hence that
$n$ divides $\binom{n}{j}$ by Lemma~\ref{Kummer}. Therefore $n$ divides $f_i(n)$ for $1\le i\le \lambda$. 
It follows that $\prod_{i=1}^n xk^{n-i}=1$, and so $k\in \W(G)$. This proves part~(a).

(b)~Let $d$ be the largest integer such that $\exp(\gamma_d(G))=\exp(G)$.
Let $k \in \B(\gamma_d(G))$. Then $|k|=n$.
Let $R$ be any commutator of weight at least $2$ involving $k$. Since $k\in\gamma_d(G)$, we see $R\in\gamma_{d+1}(G)$ and hence $R^{n/p}=1$ (as $\exp(\gamma_{d+1}(G))<n$).

We know that $f_1(n)=n$ and $f_2(n)=\binom{n}{2}$ which are both divisible by $n$, so $R_i^{f_i(n)}=1$ for $i=1,2$. All  commutators $R_i$ for $i>2$ have weight at least $2$ and involve $k$ (otherwise they would be trivial), so they satisfy $R_i^{n/p}=1$  by the argument above.

 All the commutators $R_i$ for $1\leq i\leq \lambda$  have load  $m_i\leq 2c-1<p^2$. Since the binomials $\binom{n}{j}$ involved in the formula~\eqref{E:fn} for $f_i(n)$ have $1\le j<p^2$, the $p$-part of $j$ is either $1$ or~$p$. 
In the former case $n$ divides $\binom{n}{j}$ and in the latter case $n/p$ divides $\binom{n}{j}$ by Lemma \ref{Kummer}. In both cases $n/p$ divides $\binom{n}{j}$. It follows that $n/p$ divides $f_i(n)$, and so $R_i^{f_i(n)}=1$ for $2<i\leq \lambda$.
It follows that $\prod_{i=1}^n xk^{n-i}=1$, and so $k\in \W(G)$.
\end{proof}

The conclusion that $B(G)\subseteq W(G)$ in Theorem~\ref{T7}(a) need not hold if $p\leq 2c-1$. For example take $G$ to be the extraspecial group
$3_{+}^{1+2r}$ in Remark~\ref{re:ES}, or $\textup{SmallGroup}(3^4,12)$ or $\textup{SmallGroup}(3^5,25)$. These examples have $p=3$ and $c=2, 2, 4$ and $\B(G)\not\subseteq\W(G)$, and $p^2>2c-1$, so $\B(\gamma_d(G))\subseteq\W(G)$ holds. 
All the groups in the SmallGroup database~\cites{Magma,GAP} of order dividing $3^7$ with nilpotency class at least $5$ satisfy $\B(\gamma_d(G))=\W(G)$. Theorem \ref{T8} below implies that we need to look at nilpotency class at least $7$ to find a $3$-group where $\B(\gamma_d(G))\not\subseteq\W(G)$. We were able to find an example of order $3^9$ and nilpotency class $7$ where $\B(\gamma_d(G))\not\subseteq\W(G)$, see Remark~\ref{re:c=7}.

\begin{corollary}
  Let $G$ be a nilpotent  group of odd order with nilpotency class at most~$c$.
  If all primes $p$ dividing $|G|$ satisfy $p^2>2c-1$, then $G$ is a $J$-group.
\end{corollary}

\begin{proof}
Since $G$ is a direct product of $p$-groups
(Remark~\ref{pgroupssufficient}), and each of these
is a $J$-group by Theorem~\ref{T7}(b), the result follows from Lemma~\ref{L-directproduct}.
\end{proof}
 
If $G$ has odd exponent $p$, then we already know $G$ is a $J$-group by Corollary \ref{L8}, but now we can say more about witnesses. As all non-trivial elements are big, and all non-trivial $\gamma_i(G)$ have exponent $p$, we get the following corollary.

\begin{corollary}\label{E5}
  Let $G$ be a $p$-group of exponent $p>2$ and  of nilpotency class~$c$.
  \begin{enumerate}[{\rm (a)}]
\item If $p$ is larger than $2c-1$, then  $\W(G)=G\setminus\{1\}$.
\item If $p^2$ is larger than $2c-1$, then  $\gamma_c(G)\setminus\{1\}\subseteq\W(G)$.
\end{enumerate}
 \end{corollary}

\begin{remark}\label{re:ES}
The extraspecial $p$-group $G$ of exponent $p$ has $c=2$ and satisfies $p>2c-1$ if $p>3$. Thus $\W(G)=G\setminus\{1\}$ holds by Corollary~\ref{E5}(a) if $p>3$. If $p=3$, then it can be shown that $\gamma_2(G)\setminus\{1\}=\Z(G)\setminus\{1\}=\W(G)$ has size $2$, so
Corollary~\ref{E5}(b) can not be improved. It is striking that for $p\ge 5$ the witness set $\W(G)$ is as large as possible, and for $p=3$ it is tiny.
\end{remark}

\subsection{Small nilpotency class}\label{S4.1}
In this section we computationally determine the coefficients in the formulas for $f_i(n)$ from Lemma~\ref{L-Hall} for $p$-groups with nilpotency class at most~6 and deduce  that all such groups are $J$-groups. 

Let $G$ be a $p$-group with nilpotency class $c$. Fix fixed elements $k,x\in G$ and define $R_1,\dots,R_\lambda$, $\lambda=\lambda(c)$ as in
Lemma~\ref{L-Hall}. We initially view $n\in\{0,1,2,\dots\}$ as a variable, and later take $n$ to be the exponent of $G$.
Let $w_n\coloneq \prod_{i=1}^n x k^{n-i}$, that is  $w_0=1$ and $w_n=xk^{n-1}w_{n-1}$ for $n\ge1$. Lemma \ref{L-Hall} implies that
\[
w_n=\prod_{i=1}^n xk^{n-i}= R_1^{f_1(n)}R_2^{f_2(n)}\cdots R_\lambda^{f_\lambda(n)}
\]
where $R_1,\dots,R_\lambda$ are all the commutators of weight at most $c$ in our chosen ordering.
Furthermore, $f_i(n)=nP_i(n)$ where $m_i!P_i(n)$ is a polynomial of degree at most $m_i-1$ with integer coefficients. 
Let $m$ be the largest load of an $R_i$ ($1\leq i\leq \lambda$).
Commutator collection as explained in Lemma \ref{L-Hall} can be used to compute the above factorisation for $w_j$ ($1\le j\le m$). Hence we can find the (integer-valued) exponent $f_i(j)$ of each $R_i$ for $w_j$.
Since $f_i(n)=a_{i1} \binom{n}{1}+a_{i2} \binom{n}{2}+\cdots +a_{im_i}\binom{n}{m_i}$ we may determine the coefficients $a_{ij}$ recursively as follows: 
\[
  a_{i1}=f_i(1)\quad\textup{and}\quad a_{ij}=f_i(j)-\sum_{\ell=1}^{j-1} a_{i\ell}\binom{n}{\ell}\quad\textup{for $1<i\le m_i$.}
\]
Thus to determine all the $a_{ij}$ and hence the polynomial $f_i(n)$, we need to know the values of $f_i(1), f_i(2), \dots, f_i(m_i)$, that is we need to know the exponent of $R_i$ in $w_1,w_2,\dots, w_{m_i}$.  

When $c=6$ there are 23 choices for $R_i$ (see Table~\ref{commutators}) and the {\sc Magma} code in~\cite{G} computes the 23 polynomials $f_i(n)$ and the coefficients $a_{ij}$. The 23 commutators $R_1,\dots,R_{23}$ are listed in the order they naturally arise when performing Hall's algorithm described in Lemma \ref{L-Hall}, and $R_i$ appears with the exponent  $ f_i(n)=\sum_{j=1}^{m_i}a_{ij} \binom{n}{j}$, where $m_i$ is the load of $R_i$. 

\begin{table}[!ht]
  \centering\caption{Coefficients $a_{ij}$ with $\prod_{i=1}^n xk^{n-i}\equiv \prod_{i=1}^\lambda R_i^{f_i(n)}\pmod{\gamma_{c+1}(G)}$ and $f_i(n)=\sum_{j=1}^{m_i}a_{ij}\binom{n}{j}$. The commutator $R_i$ has load $m_i$.}
  \label{commutators} \fontsize{10}{11}\selectfont
\begin{tabular}{|c|p{23.5mm}|c|ccccccccccc|}
\hline 
$i$&$R_i$& $m_i$&$a_{i1}$&$a_{i2}$&$a_{i3}$&$a_{i4}$&$a_{i5}$&$a_{i6}$&$a_{i7}$&$a_{i8}$&$a_{i9}$&$a_{i10}$&$a_{i11}$\\
  \hline
1&$x$	&1	&1&&&&&&&&&&\\								2&$k$	&2&	0&	1&&&&&&&&&\\							3&$[k,x]$&	3&	0&	1&	2&&&&&&&&\\						4&$[k,x,x]$&	4	&0&	0&	2&	3&&&&&&&\\			
5&$[k,x,k]$	&5&	0&	0&	6&	18&	12&&&&&&\\				6&$[k,x,x,x]$&	5&	0&	0&	0	&3&	4&&&&&&\\			7&$[k,x,x,k]$	&6&	0&	0&	3&	27&	54&	30&&&&&\\		8&$[k,x,k,k]$	&7	&0&	0&	2	&51&	184&	225&	90&&&&\\				
9&$[k,x,x,x,x]$	&6	&0&	0&	0&	0&	4&	5	&&&&&\\		10&$[k,x,x,x,k]$&	7	&0&	0&	0&	12&	72&	120&	60&&&&\\				
11&$[k,x,x,k,k]$&	8&	0&	0&	1&	54&	378&	910&	900&	315&&&\\			
12&$[k,x,k,k,k]$&	9&	0&	0&	0&	45&	600&	2325&	3870&	2940&	840&&\\		
13&$[k,x,x,[k,x]]$&	7&	0&	0&	4&	63&	220&	280&	120&&&&\\				
14&$[k,x,k,[k,x]]$&	8&	0&	0&	17&	291&	1394&	2800&	2520&	840&&&\\			
15&$[k,x,x,x,x,x]$&	7&	0&	0&	0&	0&	0	&5	&6&&&&\\				
16&$[k,x,x,x,x,k]$&	8&	0&	0&	0&	0&	30&	150&	225&	105&&&\\			
17&$[k,x,x,x,k,k]$&	9&	0&	0&	0&	19&	324&	1540&	3020&	2625&	840&&\\		
18&$[k,x,x,k,k,k]$&	10&	0&	0&	0&	45&	990&	6150&	16650&	22365&	14700&	3780&\\
19&$[k,x,k,k,k,k]$&	11&	0&	0	&0	&18&	912&	9015&	35946&	72030&	77280&	42525&	9450\\
20&$[k,x,x,x,[k,x]]$&	8&	0&	0&	0&	24&	228&	645&	720&	280&&&\\			
21&$[k,x,x,k,[k,x]]$	&9&	0&	0	&7&	297&	2610&	9010&	14670&	11340&	3360&&\\		
22&$[k,x,k,k,[k,x]]$&	10	&0	&0	&7&	555	&6898&	33115&	77970&	96565&	60480&	15120	&\\
23&$[k,x,k,[k,x,x]]$	&9&	0&	0	&10&	264&	2004&	6640&	10770&	8400&	2520&&\\		
\hline
\end{tabular}
\end{table}

\begin{theorem}\label{T8}
Let $G$ be a $p$-group with nilpotency class $c$ and $p>2$. 
Let $d$ be the largest integer such that $\exp(\gamma_d(G))=\exp(G)$.
If $d\geq c-5$, then $G$ is a $J$-group and $\B(\gamma_d(G))\subseteq\W(G)$. 
\end{theorem}
\begin{proof}
Let $n=\exp(G)$. Then $n$ is a $p$-power. Recall that $g^n=1$ for all $g\in G$.

Let $k \in \B(\gamma_d(G))$. Then $|k|=n$, and for each $g\in G$ we know that
$[k,g]^{n/p}=1$ since $[k,g]\in\gamma_{d+1}(G)$. As all commutators of weight at least $7$ contain at least one $k$, they lie in $\gamma_{d+6}(G)$. Now $d+6\ge c+1$, and hence  all these commutators are trivial. So we need only consider the 23 basic commutators in Table~\ref{commutators}.
The largest load of such a commutator is $2\times6-1=11$ for $[k,x,k,k,k,k]$. Hence $\textup{deg}(f_i(n))\le 11$ for $i\le 23$.

Consider the exponent $f_i(n)=\sum_{j=1}^{m_i}a_{ij}\binom{n}{j}$ of the commutator $R_i$  in Table~\ref{commutators} where $i\le23$. All $j$s in the sum satisfy  $1\le j\le 11$.
If $p^\ell\mid j$, then $\ell\le 1$ or $(p,\ell)=(3,2)$. 
By Lemma~\ref{Kummer} one of the following holds:
\begin{enumerate}[(i)]
    \item $(j,n)=1$, and $n$ divides $\binom{n}{j}$, or 
    \item $(j,n)=p$, and $n/p$ divides $\binom{n}{j}$, or
    \item $p=3$, $j=9$, and $n/9$ divides $\binom{n}{j}$.
\end{enumerate}

Using similar arguments as in the proof on Theorem \ref{T7}(b), we get that $R_i^{f_i(n)}=1$ if $m_i< 9$.
Moreover, for $9\leq m_i\leq 11$, all binomial coefficients collapse except for $\binom{n}{9}$, so $R_i^{f_i(n)}=R_i^{a_{i9}\binom{n}{9}}$. More precisely
$\prod_{i=1}^n xk^{n-i}$ equals
\begin{align*}&[k,x,k,k,k]^{840\binom{n}{9}}[k,x,x,x,k,k]^{840\binom{n}{9}} [k,x,x,k,k,k]^{14700\binom{n}{9}}\times\\
&[k,x,k,k,k,k]^{77280\binom{n}{9}}[k,x,x,k,[k,x]]^{3360\binom{n}{9}} [k,x,k,k,[k,x]]^{60480\binom{n}{9}}[k,x,k,[k,x,x]]^{2520\binom{n}{9}}.
\end{align*}
By the argument above, $n/9$ divides $\binom{n}{9}$.
It just so happens that all coefficients $a_{i9}$ in Table~\ref{commutators} are multiples of $3$, and therefore $a_{i9}\binom{n}{9}$ is divisible by $n/3$.
Since any commutator of weight at least $2$ involving $k$ has order dividing $n/3$, it follows that $\prod_{i=1}^n xk^{n-i}=1$.
\end{proof}
The following corollary follows immediately since $d\geq 1\geq c-5$ when $c\leq 6$.
\begin{corollary}
  Let $G$ be a $p$-group with nilpotency class at most~$6$ and $p>2$, and let $d$ be the largest integer such that $\exp(\gamma_d(G))=\exp(G)$.
Then $G$ is a $J$-group and $\B(\gamma_d(G))\subseteq\W(G)$.
\end{corollary}

\begin{remark}\label{re:c=7}
If $p\ge5$, then $\B(\gamma_d(G))\subseteq\W(G)$ holds for $c\le12$ by Theorem \ref{T7}(b), and it $p=3$, it holds for $c\le5$.
Consider the proof of Theorem~\ref{T8} when $p=3$, $c=7$ and $d=1$. The  41 choices for $R_i$ of weight at most 7 are listed in~\cite{G}.
The polynomial $f_i(n)=\sum_{j=1}^{m_i}a_{ij}\binom{n}{j}$ has degree at most $m_i\le 2\times 7-1=13$. If $1\leq j\le 13$ and $p^\ell\mid j$, then $\ell\leq 1$ or $(p,\ell)=(3,2)$. Then
$R_i^{f_i(n)}=R_i^{a_{i9}\binom{n}{9}}=R_i^{b_i\binom{n}{9}}$ where $a_{i9}\equiv b_i\pmod3$ and $b_i\in\{0,1,2\}$. 
We use~\cite{G} to output the $a_{i9}$ coefficients for all $R_i$ of weight at most $7$, to~get 
\begin{align*}\prod_{i=1}^n xk^{n-i}=&[k,x,k,k,k,k,k]^{\binom{n}{9}}[k,x,x,x,x,[k,x]]^{2\binom{n}{9}} [k,x,x,k,k,[k,x]]^{\binom{n}{9}}\times\\
&[k,x,x,[k,x],[k,x]]^{2\binom{n}{9}}[k,x,k,[k,x],[k,x]]^{\binom{n}{9}}. 
\end{align*}
We now exhibit such a group for which $\B(\gamma_d(G))\not\subseteq\W(G)$.
Let $x,k\in\Sym(27)$ where
\[
k= (1,10,26,9,16,23,4,13,20)(2,11,27,7,17,24,5,14,21)(3,12,25,8,18,22,6,15,19)
\]
and 
\[x= (1,19,10,7,27,16,4,22,13)(2,20,11,8,25,17,5,23,14)(3,21,12,9,26,18,6,24,15).
\]
The group $G=\langle k,x\rangle$ has order $3^9$, nilpotency class $7$, and $d=1$ as $\textup{exp}(G)=9$ and $\textup{exp}(G')=3$. Obviously $k\in\B(\gamma_d(G))=\B(G)$ but $k$ is not a witness (as can be seen with {\sc GAP} or {\sc Magma}). More precisely, $\W(G)\subsetneqq\B(G)$ and $|\W(G)|=7290$ and $|\B(G)|=16038$. Thus the hypothesis $d\ge c-5$ is needed in Theorem~\ref{T8}.
\end{remark}

\subsection{Powerful groups}

For any $p$-group $G$ and $k\geq 0$, we define $k$th omega and agemo subgroups as follows
\[
\Omega_k(G)=\langle g\in G\mid g^{p^k}=1\rangle\qquad\textup{and}\qquad
\mho^k(G)=\langle g^{p^k}\mid g\in G\rangle.
\]
Suppose that $p>2$. A $p$-group $G$ is called \emph{powerful}, if the derived subgroup $G'=\gamma_2(G)$ is contained in $\mho^1(G)$.
In particular, abelian $p$-groups are powerful and a group of exponent $p$ is powerful if and only if it is abelian.

By Lemma \ref{L-sameexp} any big element in an abelian $p$-group is a witness. The following is a significant generalisation of this fact.

\begin{theorem} Let $G$ be a powerful $p$-group, for an odd prime $p$. Then $G$ is a J-group and $\B(G)\subseteq\W(G)$. 
\end{theorem}

\begin{proof} 
 Let $n=p^e=\exp(G)$. Our proof is by induction on $e$.

Assume $e=1$. A powerful $p$-group with exponent $p$ is abelian, so $\B(G)=\W(G)$ by Lemma \ref{L-sameexp}. 
Assume now that $e\geq 2$, and that the statement is true for every  powerful $p$-group with exponent $p^{e-1}=n/p$.

Let $N\coloneq \mho^{e-1}(G)=\langle g^{n/p}\mid g \in G\rangle$, which is normal in $G$.  As $G$ is powerful, we have $N = \{g^{n/p}\mid g\in G\}$ by~\cite{LM}*{Proposition 1.7} and in particular, $\exp(N)=p$. Moreover, $G/N$ is clearly  powerful~\cite{LM}*{p.\;486}. 
For $g\in G$, $(gN)^{n/p}=g^{n/p}N=N$, so the order of $gN$ divides $n/p$. Now assume $k\in\B(G)$ and that $kN$ has order less than $n/p$, then $k^{n/p^2}\in N$, and $N$ contains an element of order $p^2$, a contradiction; thus $kN$ has order $n/p$, $\exp(G/N)=n/p$, and $kN\in\B(G/N)$. By the inductive hypothesis, $kN$ is a witness for $G/N$.

As $G$ is powerful, $N$ is also powerful~\cite{LM}*{Corollary~1.2}. Since $\exp(N)=p$, it follows that $N$ is abelian (as $\mho^1(N)\supseteq N'$). Thus $k^{n/p}$ is central in $N$.

Now Lemma~\ref{L7bis} implies that $(G,k)$ is a $J$-group, and the statement is proved. 
\end{proof}

\subsection{Power-closed groups}
Recall that a {\it section} of a group is a quotient of a subgroup. A $p$-group $G$ is called:
\begin{itemize}
\item a \emph{$P_1$-group} or \emph{power-closed} if in all sections of $G$, products of $p^k$-th powers are $p^k$-th powers for all $k\geq 1$,
\item a \emph{$P_2$-group} or \emph{order-closed} if in all sections of $G$, products of elements of order at most $p^k$ are elements of order at most $p^k$ for all $k\geq 0$,
\item \emph{regular} if, for every $a, b$ in $G$, there is a $c$ in $\gamma_2(\langle a,b \rangle)$ such that $a^p  b^p = (ab)^p c^p$.
\end{itemize}
In any order-closed $p$-group $G$, the set of elements of order at most $p^k$ equals $\Omega_k(G)$, while in any power-closed $p$-group $G$,  the set of $p^k$-th powers equals $\mho^k(G)$. By definition, the family of order-closed $p$-groups and the family of power-closed $p$-groups are both closed under taking sections.

\begin{remark}\label{rmk-groups} 
By~\cite{Ma-ps}*{p.\;121 and Corollary~4}, 
regular implies order-closed which, in turn, 
implies power-closed. 
\end{remark}

\begin{theorem}\label{T:power-closed}
 Let $G$ be a power-closed $p$-group with $p$ an odd prime. 
Then $G$ is a $J$-group and $\B(\gamma_d(G))\subseteq\W(G)$, where $d$ is the largest integer with $\exp(\gamma_d(G))=\exp(G)$.
\end{theorem}

\begin{proof}
Let $n=p^e=\exp(G)$, let $c$ be the nilpotency class of $G$, and let $d$ be the largest integer such that $\exp(\gamma_d(G))=\exp(G)$.
Our proof is by induction on $e$.

Assume $e=1$. Then $p=\exp(\gamma_c(G))=\exp(G)$, thus $d=c$. Note $\gamma_c(G)\leq \Z(G)$. By Corollary \ref{L8}, $G$ is a $J$-group and $\Z(G)\cap \W(G)=\B(\Z(G))$.
It therefore follows that $\B(\gamma_d(G))\subseteq \B(\Z(G))\subseteq \W(G)$, and the result is proved for $e=1$.

Assume $e\geq 2$, and that the result is true for power-closed $p$-groups of exponent~$p^{e-1}$.

We define $N_1\coloneq \mho^{e-1}(G)=\langle\{g^{n/p}\mid g \in G\}\rangle$. Then $1<N_1\lhdeq G$ and $\textup{exp}(N_1)=p$.  As $G$ is power-closed, $N_1 = \{g^{n/p}\mid g\in G\}$.
We also define $N_2\coloneq \mho^{e-2}(\gamma_{d+1}(G))$, a characteristic subgroup of $\gamma_{d+1}(G)$ and hence a normal subgroup of $G$. As $G$ is power-closed, $N_2$ consists of $p^{e-2}$-th powers of elements from $\gamma_{d+1}(G)$, but as $\gamma_{d+1}(G)$ has exponent at most $p^{e-1}$, $N_2$ is of exponent at most $p$. By \cite{Ma-ps}*{Theorem~13}, $N_1$ and $N_2$ commute since $N_1\leq \mho^1(G)$ when $e\geq 2$. Hence, $N\coloneq N_1\cdot N_2$ is a normal subgroup of exponent $p$ in $G$. But then again by \cite{Ma-ps}*{Theorem~13}, $N$ and $\mho^1(G)$ commute.
 Moreover, $G/N$ is a section of $G$ so is power-closed too. 

We now claim that, for any subgroup $M$ of $G$ with $\exp(M)=\exp(G)=n$, we have $\exp(MN/N)=n/p$ and $bN\in\B(MN/N)$ for all $b\in \B(M)$. Indeed, for $m\in M$, $(mN)^{n/p}=m^{n/p}N=N$ since $m^{n/p}\in N_1\leq N$, so the order of $mN$ divides $n/p$. Assume that $b\in\B(M)$ and that $bN$ has order less than $n/p$. Then $b^{n/p^2}\in N$, and $N$ contains an element of order $p^2$, a contradiction; thus $bN$ has order $n/p$, $\exp(MN/N)=n/p$, and $bN\in\B(MN/N)$.

Let $k\in \B(\gamma_d(G))\subseteq \B(G)$. Then $kN$ has order $n/p$. Taking $M=G$, it follows from the claim that $\exp(G/N)=n/p$. Moreover, $k^{n/p}$ is central in $N$, since $N$ and $\mho^1(G)$ commute. It suffices, by Lemma \ref{L7bis}, to show that $(G/N,kN)$ is a $J$-group, as then $(G,k)$ is a $J$-group and the result follows. By the inductive hypothesis, it suffices to prove  that $kN\in \B(\gamma_{d'}(G/N))$, where $d'$ is the largest integer such that $\exp(\gamma_{d'}(G/N))=\exp(G/N)$. Note that $\gamma_d(G/N)=\gamma_d(G)N/N$ has exponent $n/p$ by the claim above, so $d'\geq d$. 
Assume $\gamma_{d+1}(G/N)=\gamma_{d+1}(G)N/N$ has exponent $n/p$ too. Then there exists $t\in \gamma_{d+1}(G)$ such that $tN$ has order $n/p$. Hence $t^{n/p^2}\notin N$, contradicting the fact that $t^{n/p^2}\in N_2\leq N$. 
Thus $d'=d$. 
 Since  $k\in \B(\gamma_d(G))$, $kN\in \gamma_d(G)N/N=\gamma_d(G/N)$, and $kN$ has order $n/p$, so $kN\in \B(\gamma_{d}(G/N))$, which finishes the proof.
\end{proof}

\section{Non-nilpotent examples of \texorpdfstring{$J$}{}-groups}\label{S5}

In this section we study metacyclic $J$-groups that are not nilpotent.

Recall that the group $\mathbb{Z}_n^\times$ of units of $\mathbb{Z}_n$ has order $\phi(n)$. For $\alpha\in\mathbb{Z}_n$ the order of $\alpha$ modulo~$n$ is the least $i\ge1$ with $\alpha^i\equiv1\pmod n$ and is denoted $\textup{ord}_n(\alpha)$.

We consider metacyclic groups of the form $\C_s\rtimes\C_r$ where $rs$ is odd. Let $\alpha\in\{1,\dots,s\}$ satisfy and $\alpha^r\equiv 1\pmod s$ (so $\textup{ord}_s(\alpha)$ divides $r$) and define
\begin{equation}\label{E:Mrs}
  M_{r,s,\alpha}=\langle a,b\mid a^r=b^s=1, a^{-1}ba=b^\alpha\rangle.
\end{equation}
Now $M_{r,s,1}$ is abelian and hence a $J$-group by Lemma~\ref{L-sameexp}(b). We assume henceforth that $1<\alpha<s$, and we abbreviate $a^{-1}ba$ by $b^a$ and note that $a^{-\ell}b a^\ell=b^{\alpha^\ell}$ for $\ell\in\ZZ$.

Since $\langle\alpha\rangle\le\mathbb{Z}_s^\times$, Lagrange's Theorem shows that $\textup{ord}_s(\alpha)$ divides $|\mathbb{Z}_s^\times|=\phi(s)$.
On the other hand, $\alpha^r\equiv 1\pmod s$ implies that $\textup{ord}_s(\alpha)\mid r$, and hence $\textup{ord}_s(\alpha)\mid\gcd(r,\phi(s))$. 

We next determine when $M_{r,s,\alpha}$ is nilpotent (see Section \ref{S4} for the definition).

\begin{lemma}\label{Lmetanilp}
The  group $M_{r,s,\alpha}$ is nilpotent if and only if every prime factor of $s$ divides $\alpha-1$, that is, the square-free part of $s$ divides $\alpha-1$.
\end{lemma}
\begin{proof}
Set $G=M_{r,s,\alpha}$.
Using the fact that $[b,a]=b^{-1}b^a=b^{\alpha-1}$,  an easy induction shows that $\gamma_{\ell+1}(G)=\langle b^{{(\alpha-1)}^\ell}\rangle$ for $\ell>0$. Therefore $G$ is nilpotent if and only if ${(\alpha-1)}^\ell\equiv 0\pmod s$ for $\ell$ large enough, that is, if every prime factor of $s$ divides $\alpha-1$.
\end{proof}

The elements of $M_{r,s,\alpha}$ can be written uniquely as $a^ib^j$ where $0\le i<r$ and $0\le j<s$. Since $b^{j_1}a^{i_2}=a^{i_2}b^{j_2\alpha^{i_2}}$, we
have the following multiplication rule: 
\begin{equation}\label{E:mult}
a^{i_1}b^{j_1}a^{i_2}b^{j_2}=a^{i_1+i_2}b^{j_1\alpha^{i_2}+j_2}.
\end{equation}
To see that many groups $M_{r,s,\alpha}$ are $J$-groups, we use the 
following technical lemma.

\begin{lemma}\label{Lnumbertheory}
  Let $s$ be a positive odd integer, and let $\alpha\in\ZZ$ satisfy $\alpha^s\equiv 1\pmod s$. 
  Then {\rm (a)}~$\sum_{\ell=0}^{s-1}\alpha^{\ell}\equiv 0\pmod s$, and
  {\rm (b)}~$\sum_{\ell=0}^{s-1}\ell\alpha^{\ell}\equiv 0\pmod s$.
\end{lemma}
\begin{proof}
This follows from \cite{DG}*{Theorem 1}: take $k=0$ for part (a), and take $k=1$ then multiply by $\alpha$ for part (b).
\end{proof}

\begin{lemma}\label{L3bis}
Consider the group $M_{r,s,\alpha}$ given by~\eqref{E:Mrs}, where $rs$ is odd, $1<\alpha<s$ and $\alpha^r\equiv 1\pmod s$. Then  $b$ is a witness if and only if  $r$ divides $s$.  
\end{lemma}

\begin{proof}
Set $G=M_{r,s,\alpha}$.
As $b$ has order $s$, we compute $w_s\coloneq \prod_{\ell=1}^s xb^{s-\ell}$  for  $x\in G$. Collecting the $x$s to the left shows that
\[
  w_s=\prod_{\ell=1}^s xb^{s-\ell}=x^s(b^{s-1})^{x^{s-1}}\cdots (b^2)^{x^2}b^x=x^s.\prod_{\ell=1}^s (b^{s-\ell})^{x^{s-\ell}}.
\]
 Set $x=a^ib^j$. Then it follows from \eqref{E:mult} that $x^s=a^{is}b^{jm}$ where $m=\sum_{\ell=0}^{s-1}\alpha^{i\ell}$. Moreover, $b^x=b^{a^ib^j}=b^{a^i}=b^{\alpha^i}$ and so $(b^\ell)^{x^\ell}=b^{\ell\alpha^{i\ell}}$. Thus
$\prod_{\ell=1}^s (b^{s-\ell})^{x^{s-\ell}}=b^n$ where $n=\sum_{\ell=0}^{s-1}\ell \alpha^{i\ell}$. Hence $w_s$ equals $a^{is}b^{jm+n}$.

Suppose $b$ is a witness. Then $w_s=1$ for all  $0\leq i<r, 0\leq j<s$. In particular, taking $i=1$ we see that $r$ divides $s$. 

Conversely suppose $r\mid s$. It follows from $\alpha^r\equiv 1\pmod s$ that $(\alpha^i)^s\equiv 1\pmod s$. Thus, applying Lemma \ref{Lnumbertheory} to $\alpha^i$ shows that  
 $b^m=b^n=1$. Thus $w_s=a^{is}=1$ for all $i\in\{0,\dots,r-1\}$, since $\textup{ord}_s(\alpha)\mid r\mid s$.
In summary, $b\in\W(G)$ if and only if $r\mid s$.
\end{proof}

If $r\mid s$, then $M_{r,s,\alpha}$
is a $J$-group. Indeed, the witness $k=b$ and the function $f(a^ib^j)=a^{ij}b^{n(j)}$ where $n(j)=\sum_{\ell=0}^{j-1}\ell\alpha^{i\ell}$ satisfy \eqref{definingidentity}. Note that, as $b$ is a big element, $\B(G)\cap\W(G)$ is non-empty for this family of $J$-groups.

This allows us to find non-nilpotent examples of $J$-groups, with a big witness.

\begin{example}\label{exnonnilpJ2}
   Let $r=p$ an odd prime and let $s=pq$ where $q\equiv1\pmod p$ is a prime (there are infinitely such primes by Dirichlet's theorem). As
   $\mathbb{Z}_s^\times\cong\mathbb{Z}_p^\times\times\mathbb{Z}_q^\times$ has order $\phi(s)=(p-1)(q-1)$, there exists an $\alpha\in\mathbb{Z}_s^\times$ such that $\textup{ord}_s(\alpha)=p$. 
  Therefore, $\alpha^p\equiv 1\pmod p$. On the other hand $\alpha^{p-1}\equiv 1\pmod p$, and so $\alpha\equiv 1\pmod p$. 
  However, $\alpha\not\equiv 1\pmod s$ as $\textup{ord}_s(\alpha)=p$, so the Chinese Remainder Theorem implies that $\alpha\not\equiv 1\pmod q$.
   Thus $M_{r,s,\alpha}$ is a $J$-group by Lemma \ref{L3bis}, and it is non-nilpotent by Lemma \ref{Lmetanilp}.
 Thus the following groups are non-nilpotent $J$-groups: $M_{3,21,4}$,  $M_{3,39,16}$, $M_{5,155,16}$. 
\end{example}

In previous sections we concentrated on finding witnesses among big elements, see Question~\ref{Q2}. 
Now we can find an example of a $J$-group with a non-big witness.

\begin{example}\label{nonbigwitness}
Let $G$ be the $3$-group $M_{3,9,4}$. Clearly $G$ has exponent $9$. We show that $k=a^2$ is a witness of order $3$. To see this
it suffices to check by Lemma~\ref{L11} and \eqref{E:mult}, that $xk^2xkx=1$ holds for the double coset representatives $x\in\{1,b\}$. This is indeed true.
\end{example}

\section*{Acknowledgments}

SG and DB acknowledge the support of the Australian Research Council Discovery Grant DP190100450. TB acknowledges the support of 
projects P29931 and AR561 of the Austrian Science Foundation FWF. DB acknowledges support by the RWTH Aachen University Scholarship for Doctoral Students.
The research of JF was partially funded by the Deutsche Forschungsgemeinschaft (DFG, German Research Foundation) – Project-ID 286237555 – TRR 195.
We thank Mike Newman for his helpful comments and for verifying that $3$-groups of order dividing $3^{10}$ are $J$-groups.
This research project arose from a problem posed at the CMSC Annual Research
Retreat in 2018. We thank the CMSC for its support.
Finally, we thank the referee for their
careful reading of the paper and helpful comments, and Rafael Dahmen
for spotting a minor error in Lemma~\ref{L9}.

\end{document}